\documentclass[11pt]{article}

\usepackage[margin=1in]{geometry}
\usepackage{amsmath,amssymb,amsthm}
\usepackage{microtype}
\usepackage[colorlinks=true,citecolor=blue,urlcolor=blue]{hyperref}

\newtheorem{theorem}{Theorem}
\newtheorem{lemma}[theorem]{Lemma}
\newtheorem{corollary}[theorem]{Corollary}
\newcommand{\Z}{\mathbb Z}
\newcommand{\Rcal}{\mathcal R}
\newcommand{\diamaff}{\operatorname{diam}}

\title{A Local Classification of Four-Element Multiple Sumsets}
\author{Minkyu Jung\\
Independent Researcher, Republic of Korea\\
\texttt{minkyu.jung3@gmail.com}}
\date{21 July 2026}

\begin{document}
\maketitle

\begin{abstract}
For a finite set \(A\subset\Z\), write \(hA\) for its \(h\)-fold
sumset, and let
\[
  \Rcal(h,k)=\{|hA|:A\subset\Z,\ |A|=k\}.
\]
We determine the part of \(\Rcal(h,4)\) lying between \(4h+2\) and
\(6h-4\): for \(h=4\) the only value is \(5h-1\), while for \(h\ge5\)
the only values are \(5h-1\) and \(5h+1\).  This proves Rajagopal's
conjectured gap \(5h\notin\Rcal(h,4)\) for every \(h\ge4\).
For \(h\ge6\), it also yields the new missing interval
\([5h+2,6h-4]\), which lies outside Rajagopal's general excluded set.
Lev's lower bound for the successive growth of multiple sumsets reduces the
problem to normalized sets of affine diameter five, of which there are only
six.  Reflection and four elementary exact sumset computations finish the
classification.
\end{abstract}



\section{Introduction}

For an integer \(h\ge1\) and a finite set \(A\subset\Z\), define
\[
 hA=\{a_1+\cdots+a_h:a_1,\ldots,a_h\in A\}.
\]
Nathanson introduced the set \(\Rcal(h,k)\) of possible cardinalities
\(|hA|\) with \(|A|=k\) and posed the problem of determining it
\cite{Nathanson}.  Tang and Xing \cite{TangXing}, Schinina
\cite{Schinina}, and Rajagopal \cite{Rajagopal} established general
intervals missing from these ranges.  In particular, for \(h\ge4\),
Rajagopal's theorem gives
\[
 \Rcal(h,4)\cap[4h+2,5h-2]=\varnothing.
\]
Rajagopal then recorded in Section~5.2 the further conjectural gap,
supported by computation,
\[
 5h\notin\Rcal(h,4)\qquad(h\ge4).
\]

We prove more: in the larger interval \([4h+2,6h-4]\), we determine every
possible value.  The argument combines Lev's multiple-addition estimate
with a complete analysis at a single affine diameter.  It is short, but one
point is essential: the diameter in Lev's theorem is invariant under
nonzero integral dilations.  It is therefore the affine diameter, rather
than merely \(\max A-\min A\).
In particular, besides resolving the conjectured point \(5h\), the
classification rules out every integer in \([5h+2,6h-4]\) for \(h\ge6\);
these integers lie outside the general excluded set
\(\Delta_{h,4}\) from \cite{Rajagopal}.

\begin{theorem}\label{thm:main}
For every integer \(h\ge4\),
\[
 \Rcal(h,4)\cap[4h+2,6h-4]
 =
 \begin{cases}
  \{5h-1\},&h=4,\\
  \{5h-1,5h+1\},&h\ge5.
 \end{cases}
\]
\end{theorem}

\begin{corollary}[Rajagopal's conjectured gap]\label{cor:rajagopal}
For every integer \(h\ge4\), one has \(5h\notin\Rcal(h,4)\).
\end{corollary}

\section{Normalization and the diameter reduction}

For a nonempty finite set \(A\subset\Z\), put
\[
 g(A)=\gcd\{a-\min A:a\in A\}
\]
and define its affine diameter by
\[
 \diamaff(A)=\frac{\max A-\min A}{g(A)}.
\]
Equivalently, \(\diamaff(A)+1\) is the minimum number of terms in an
arithmetic progression containing \(A\).
Indeed, the common difference of every arithmetic progression containing
\(A\) divides \(g(A)\), while the progression with common difference
\(g(A)\) and endpoints \(\min A,\max A\) contains \(A\).

\begin{lemma}[Normalization]\label{lem:normalization}
Let \(m=\min A\), \(g=g(A)\), and \(B=(A-m)/g\).  Then
\[
 \min B=0,\qquad \gcd\{b:b\in B\}=1,\qquad
 \max B=\diamaff(A),
\]
and \(|hB|=|hA|\) for every \(h\ge1\).
\end{lemma}

\begin{proof}
The first three assertions follow directly from the definitions.  The map
\[
 hA\longrightarrow hB,\qquad s\longmapsto\frac{s-hm}{g}
\]
is a bijection, with inverse \(t\mapsto gt+hm\).
\end{proof}

We use the following sharp growth estimate of Lev \cite{Lev}.

\begin{theorem}[Lev]\label{thm:lev}
Suppose that \(B\subset\Z\), \(|B|=k\ge3\),
\(\min B=0\), \(\gcd\{b:b\in B\}=1\), and \(\max B=d\).
For every \(i\ge2\),
\[
 |iB|-|(i-1)B|\ge
 \min\{d,i(k-2)+1\}.
\]
\end{theorem}

\begin{lemma}[Diameter reduction]\label{lem:diameter}
Let \(A\subset\Z\), \(|A|=4\), and \(h\ge4\).  If
\[
 4h+2\le |hA|\le6h-4,
\]
then \(\diamaff(A)=5\).
\end{lemma}

\begin{proof}
Normalize \(A\) by Lemma~\ref{lem:normalization} and write
\[
 B=\{0,a,b,d\},\qquad 0<a<b<d.
\]
If \(d\le4\), then \(hB\subset[0,hd]\), so
\[
 |hA|=|hB|\le hd+1\le4h+1,
\]
contrary to the lower bound on \(|hA|\).

Suppose instead that \(d\ge6\).  Theorem~\ref{thm:lev} with \(k=4\)
gives
\[
 |2B|\ge |B|+\min\{d,5\}\ge9.
\]
For each \(i\ge3\), it also gives
\[
 |iB|-|(i-1)B|\ge\min\{d,2i+1\}\ge6.
\]
Therefore
\[
 |hA|=|hB|\ge9+6(h-2)=6h-3,
\]
contrary to the upper bound on \(|hA|\).
The only remaining integer value of \(d\) is \(5\).
\end{proof}

\section{The affine-diameter-five sets}

In this section, every interval is an interval of integers.
A normalized four-element set of affine diameter five has the form
\(\{0,a,b,5\}\), with
\(\{a,b\}\in\binom{\{1,2,3,4\}}2\).  Thus there are exactly six:
\[
\begin{split}
&\{0,1,2,5\},\quad \{0,1,3,5\},\quad \{0,1,4,5\},\\
&\{0,2,3,5\},\quad \{0,2,4,5\},\quad \{0,3,4,5\}.
\end{split}
\]
All six have gcd one, since their gcd divides \(5\) but each contains an
interior element not divisible by \(5\).

\begin{lemma}[Reflection]\label{lem:reflection}
For any finite \(A\subset\Z\) and any \(c\in\Z\),
\[
 h(c-A)=hc-hA.
\]
In particular, \(|h(c-A)|=|hA|\).
\end{lemma}

\begin{proof}
The identity follows by writing
\[
 (c-a_1)+\cdots+(c-a_h)=hc-(a_1+\cdots+a_h).
\]
The map \(x\mapsto hc-x\) is a bijection.
\end{proof}

Reflection about \(5\) pairs
\[
\{0,1,2,5\}\longleftrightarrow\{0,3,4,5\},\qquad
\{0,1,3,5\}\longleftrightarrow\{0,2,4,5\},
\]
and fixes \(\{0,1,4,5\}\) and \(\{0,2,3,5\}\).
It therefore suffices to calculate four multiple sumsets.

\begin{lemma}\label{lem:four}
For the four reflection representatives, the following identities hold:
\begin{align}
h\{0,1,2,5\}
  &= [0,5h-3]\cup\{5h\} &&(h\ge1), \label{eq:a}\\
h\{0,1,4,5\}
  &= [0,5h] &&(h\ge3), \label{eq:b}\\
h\{0,2,3,5\}
  &= [0,5h]\setminus\{1,5h-1\} &&(h\ge1), \label{eq:c}\\
h\{0,2,4,5\}
  &= [0,5h]\setminus\{1,3\} &&(h\ge1). \label{eq:d}
\end{align}
\end{lemma}

\begin{proof}
For \eqref{eq:a}, induction starts at \(h=1\).  If the formula holds
at \(h\), the translates by \(0,1,2,5\) of the interval
\([0,5h-3]\), together with the translates of \(5h\), have union
\[
 [0,5h+2]\cup\{5h+5\},
\]
which is the formula at \(h+1\).

For \eqref{eq:b}, observe that
\[
 \{0,1,4,5\}=\{0,1\}+\{0,4\}.
\]
Consequently
\[
 h\{0,1,4,5\}
 = [0,h]+4[0,h]
 = \bigcup_{j=0}^h[4j,4j+h].
\]
When \(h\ge3\), consecutive intervals meet or are adjacent, and their
union is \([0,5h]\).

For \eqref{eq:c}, use induction from \(h=1\).  Writing the induction
hypothesis as
\[
 \{0\}\cup[2,5h-2]\cup\{5h\},
\]
the union of its translates by \(0,2,3,5\) is
\[
 \{0\}\cup[2,5h+3]\cup\{5h+5\}
 =[0,5(h+1)]\setminus\{1,5(h+1)-1\}.
\]

Finally, for \eqref{eq:d}, the induction hypothesis can be written as
\[
 \{0,2\}\cup[4,5h].
\]
The union of its translates by \(0,2,4,5\) is exactly
\[
 \{0,2\}\cup[4,5h+5],
\]
which proves the formula at \(h+1\).
\end{proof}

\begin{proof}[Proof of Theorem~\ref{thm:main}]
Suppose that \(|A|=4\), \(h\ge4\), and
\(4h+2\le|hA|\le6h-4\).  Lemma~\ref{lem:diameter} forces affine
diameter five.  After normalization, the enumeration above and
Lemma~\ref{lem:reflection} reduce \(A\) to one of the four sets in
Lemma~\ref{lem:four}.  Their \(h\)-fold sumsets have cardinalities
\[
 5h-1,\qquad 5h+1,\qquad 5h-1,\qquad 5h-1.
\]
Thus these are the only candidates.  The two values \(5h-1\) and \(5h+1\)
are attained by \(\{0,1,2,5\}\) and \(\{0,1,4,5\}\), respectively.  When \(h=4\),
\(5h+1>6h-4\), whereas for \(h\ge5\) both values lie in
\([4h+2,6h-4]\).  This proves the stated classification.
\end{proof}

\begin{proof}[Proof of Corollary~\ref{cor:rajagopal}]
The number \(5h\) belongs to \([4h+2,6h-4]\) for every \(h\ge4\), and
it is not among the values listed in Theorem~\ref{thm:main}.
\end{proof}

\section*{Computational check}

A short independent Python program, \texttt{verify\_rajagopal.py},
accompanying this note enumerates normalized sets \(\{0,a,b,d\}\),
constructs their iterated sumsets, checks the four identities in
Lemma~\ref{lem:four} and the reflection identity, and verifies the full
classification in Theorem~\ref{thm:main} over any requested finite range.
The computation is not used in the proof.

For example, the command
\[
\texttt{python3 verify\_rajagopal.py --max-d 50 --max-h 30}
\]
checks the local classification for 449,496 normalized
set--multiplicity pairs and obtains exact agreement for every
\(4\le h\le30\).  It also checks 482,792 instances of the specialized
Lev increment bound.  These finite checks are included only as an audit of
the implementation and of the case analysis.

\section*{Acknowledgements}

The proof in this note was obtained with substantial assistance from an AI
system.  I subsequently verified every step independently, both by hand and
through the accompanying finite exhaustive computational checks over the
stated ranges, and take full responsibility for the contents.  I am grateful
to I.~Rajagopal for helpful comments and for sharing his perspective on the
geometry of possible sumset sizes.

\end{document}